\documentclass[preprint, 12pt]{elsarticle}
\usepackage{amssymb}
\usepackage{amsmath}
\usepackage{graphicx}
\usepackage{amsthm}
\usepackage{thm-restate}
\usepackage{hyperref}
\usepackage{cleveref}
\usepackage{amscd}
\usepackage{enumerate}
\usepackage{mathrsfs}
\date{}
\usepackage{tikz}
\usepackage{multicol}
\usepackage{mathtools}
\usepackage{thmtools}
\makeatletter
\def\ps@pprintTitle{%
 \let\@oddhead\@empty
 \let\@evenhead\@empty
 \let\@oddfoot\@empty
 \let\@evenfoot\@empty
}
\renewcommand\section{\@startsection {section}{1}{\z@}%
  {-3.5ex \@plus -1ex \@minus -.2ex}%
  {2.3ex \@plus.2ex}%
  {\normalfont\bfseries\Large}}

\renewcommand\subsection{\@startsection {subsection}{2}{\z@}%
  {-3.25ex\@plus -1ex \@minus -.2ex}%
  {1.5ex \@plus .2ex}%
  {\normalfont\bfseries\large}}
\makeatother

\usepackage{orcidlink}

\def \Z {{\mathbb Z}}

\def \F {{\mathbb F}}

\def \Aut{{\mbox{Aut}}}
\def \ker{{\mbox{ker}}}

\newtheorem{theorem}{Theorem}[section]
\newtheorem{cor}[theorem]{Corollary}
\newtheorem{lemma}[theorem]{Lemma}
\newtheorem{pro}[theorem]{Proposition}
\newtheorem{rem}[theorem]{Remark}

\newtheorem{obs}[theorem]{Observation}
\newtheorem{definition}[theorem]{Definition}

\makeglossary
\journal{}
\begin{document}

\begin{frontmatter}

\title{\textbf{An Arithmetic Characterization of 2-Generated Numbers}}

% Authors with labels
\author[inst1]{Bireswar Das\orcidlink{0000-0002-3758-4033}}
\ead{bireswar@iitgn.ac.in}
% \texttt{}

\author[inst2]{Kavita Samant\orcidlink{0009-0006-9402-5154}}
\ead{ks299@snu.edu.in}
% \cortext[cor1]{Corresponding author}

\author[inst3]{Dhara Thakkar\orcidlink{0000-0002-4234-0105}}
\ead{dharathakkar754@gmail.com}
% \textttt{}

% Labeled affiliation
\affiliation[inst1]{organization={Department of Computer Science and Engineering, Indian Institute of Technology},
             city={Gandhinagar},
             country={India}}
\affiliation[inst2]{organization={Department of Mathematics, Shiv Nadar Institution of Eminence},
             city={Delhi-NCR},
             country={India}}
\affiliation[inst3]{organization={
Graduate School of Mathematics, Nagoya University},
             city={Nagoya},
             country={Japan}}
             
\begin{abstract}
A group $G$ is said to be $k$-generated if it has a generating set with $k$ elements.  A positive integer $n$ is called a \emph{2-generated number} if every group of order $n$ is 2-generated. In this article,  we establish an arithmetic characterization of 2-generated numbers expressed in terms of the prime factorization of $n$.
\end{abstract}

\end{frontmatter}

\section{Introduction}
For a natural number $n$, let $\mathcal{G}_n$ denote the set of all groups of order $n$ up to isomorphism. Given a group-theoretic property $P$, we say that $n$ is a \emph{$P$-number} if every group in $\mathcal{G}_n$ satisfies property $P$. For example, if $P$ denotes the property of being cyclic, then $n$ is termed a \emph{cyclic number} if every group of order $n$ is cyclic. Various properties $P$ such as being abelian, nilpotent, satisfying the converse of Lagrange’s theorem (CLT) have been studied in this context, often yielding elegant arithmetic characterizations of the corresponding $P$-numbers (see~\cite{berger,cyclicnum, Pakia,cltnum}). In these works, the characterizations are mainly based on the prime factorizations of the integers involved. 

For instance, a positive integer $n$ is a cyclic number if and only if $\gcd(n, \varphi(n)) = 1$, where $\varphi(n)$ is Euler’s totient function (see \cite{cyclicnum}). In~\cite{Pakia}, the authors developed a unified approach to characterize \emph{cyclic numbers}, \emph{abelian numbers} and \emph{nilpotent numbers}. For their characterization, they used the notion of \emph{nilpotent factorization}, defined as follows: let $n = p_1^{e_1}p_2^{e_2} \cdots p_k^{e_k}$ be the prime factorization of $n$, where the $p_i$ are distinct primes. Then $n$ is said to have a \emph{nilpotent factorization}, if for all integers $a, i, j$ with $1 \leq a \leq e_i$, we have $p_i^a \not\equiv 1\pmod{p_j}$. 

With this definition, a positive integer $n$ is a \emph{nilpotent number} if and only if it has nilpotent factorization. Furthermore, $n$ is a cyclic number if and only if it is a square-free number with nilpotent factorization, and $n$ is an \emph{abelian number} if and only if it is cube-free and has nilpotent factorization.

In \cite{berger}, an arithmetic characterization is given for CLT-numbers, and in \cite{cltnum}, the authors further explored special cases, such as abelian and cyclic CLT-groups, which are referred to as ACLT-numbers and CCLT-numbers, respectively.

 Motivated by the interplay between group-theoretic properties and arithmetic structure, we study a new class of integers, namely \emph{2-generated numbers}, defined as follows:

 \begin{definition}\label{def-2Generation}
     A positive integer $n$ is called a \emph{2-generated number} if every group of order $n$ can be generated by two elements.
 \end{definition}
 
While the structure of 1-generated (that is, cyclic) groups is well-understood and has been extensively characterized, the study of 2-generated groups is more challenging and less structured. Numerous investigations have addressed conditions in this direction, though the results are scattered in the literature without a comprehensive characterization. A well-known result in this area is that every finite simple group is 2-generated. The proof is mainly based on the Classification of Finite Simple Groups. It begins with Miller’s early work on alternating groups~\cite{simple1}. In~\cite{simple5}, Steinberg proved this for simple groups of Lie type. However, the proof was finally completed in 1984, when Aschbacher and Guralnick~\cite{simple6} showed that every sporadic group is 2-generated. 

Apart from simple groups, several other interesting families of groups, such as symmetric, dihedral, and dicyclic groups, are also 2-generated even though their internal structures are different. The structural differences among these groups show the complexity of 2-generated groups. This makes it important to look more closely at how generation is linked to the arithmetic of the group order. This is one of the motivations for studying 2-generated numbers. 

To better understand the structure of 2-generated groups, Lucchini and Maróti introduced the concept of generating graphs (explicitly) and extensively studied in~\cite{gen1, gen2}. The \emph{generating graph} $\Gamma(G)$ of a finite group $G$ is the graph whose vertex set consists of the elements of $G$, with an edge joining two distinct elements if and only if they generate $G$. These graphs are trivial (that is, empty) for groups that cannot be generated by two elements. 
Determining the values of $n$ for which all groups of order $n$ are 2-generated yields a sufficient condition for the generating graphs to be non-empty, based solely on the number of vertices.

The main results of this article are stated as follows:
\begin{restatable}{theorem}{mainthmOdd}
\label{thm:odd}
Let $n$ be an odd positive integer. Let $n=p_1^{e_1}p_2^{e_2}\dots p_k^{e_k}$, where $p_1>p_2>\dots >p_k$ are distinct primes and each $e_i$ is a positive integer. Then $n$ is a 2-generated number if and only if $n$ is cube-free and for every $i<j,$ $p_j\nmid (p_{i}-1)$  or $e_i=1$.
\end{restatable}
\begin{restatable}{theorem}{mainthmEven}
\label{thm:even}
Let $n$ be an even positive integer. Then $n$ is a 2-generated number if and only if $n=2^{\alpha}p_1p_2p_3\dots p_k$, where $p_i's$ are odd primes and distinct, and $\alpha\leq 2$.
\end{restatable}

\paragraph*{Organization of the paper} The paper is organized into four sections. Section~\ref{sec:prelim} introduces the necessary preliminary concepts that will be used throughout the paper. Section~\ref{sec:cube} focuses on determining conditions under which a cube-free number of the form  $p^{a}q^{b}$, where $p$ and $q$ are distinct primes and $a+b\leq 3$, is $2$-generated. In  Sections~\ref{sec:main}, we prove \Cref{thm:odd} and \Cref{thm:even}.
\section{Preliminaries}\label{sec:prelim}
In this section, we recall some fundamental concepts and results of group theory, and introduce notations that will be used throughout the paper. Interested readers may refer to \cite{robin,rotman} for more details.

In this paper, we consider only finite groups. For an integer $n$, the set of prime divisors of $n$ is denoted by $\pi(n)$. For a group $G$, we denote by $|G|$ the order of $G$, and $\Aut(G)$ as the automorphism group of $G$. For a subgroup $N$ of $G$, the index of $N$ in $G$ is denoted by  $[G:N]$. For a prime $p$, $S_p$ denotes a Sylow $p$-subgroup of $G$ and $n_p$ denotes the number of Sylow $p$-subgroups of $G$. A non-trivial subgroup $N$ of a group $G$ is said to be a minimal normal subgroup if it is normal,  and for any normal subgroup $X$ of $G$ such that $\{1\}\leq X\leq N$, $X=N$ or $X=\{1\}$. Every minimal normal subgroup of a finite solvable group is an elementary abelian $p$-group, for some prime $p$ (see~\cite[Theorem 5.24]{rotman}). Let $H$ be a subgroup (not necessarily normal) of a group $G$, then a subgroup $K$ of $G$ is called a complement of $H$ in $G$ if $K\cap H=\{1\}$ and $HK=G$. Let $C_G(H)$ and $ N_G(H) $ denote the centralizer and normalizer of a subset $H$  in  $G$, respectively.

The general linear group of $n\times n$ matrices over the finite field $\F_q$, where $q=p^k$ for some prime $p$, is denoted by $\mathrm{GL}(n, q)$. The \emph{projective special linear group} of degree $n$ over a field $\F_q$ is denoted by $\mathrm{PSL}(n, q)$.

\noindent\paragraph*{\textbf{Semidirect Product}} Let $N$ and $H$ be two groups and suppose that we have an action, $\phi:H\rightarrow \Aut(N)$, of $H$ on $N$, then $G:=\{(n,h) \mid n \in N, h \in H\}$ with the product operation, $(n_1,h_1)(n_2,h_2)=(n_1\phi_{h_1^{-1}} (n_2),h_1h_2)$, for all $(n_1,h_1),(n_2,h_2)\in G$, forms a group, known as the \emph{semidirect product} of $H$ by $N$, denoted by $N \rtimes_{\phi} H$.
The following theorem is known as the Schur--Zassenhaus theorem.
\begin{theorem}[Theorem~9.1.2,~\cite{robin}]\label{thm:schur}
   Let $N$ be a normal subgroup of a finite group $G$. Suppose $\gcd(|N|, [G:N])=1$,  then $G$ contains subgroups of order $[G:N]$ and any two of them are conjugate in $G$.
\end{theorem}

\paragraph*{\textbf{Metacyclic Groups}} A group $G$ is called \emph{metacyclic} if it contains a cyclic normal subgroup $N$ such that $G/N$ is also cyclic. The subgroup $N$ is called the kernel of $G$. %Metacyclic group has been classified completely in \cite{meta}. 
A metacyclic group $G$ can be written as $G = HN$ with $H\leq G$, $N \triangleleft G$, and both $H$ and $N$ are cyclic. Such a product is called \emph{metacyclic factorization} of $G$. If $G$ has a metacyclic
factorization $HN$ with $H\cap N = \{1\}$, then $G$ is a semidirect product of $H$ by $N$. 
%It was already known to H\"{o}lder 
A finite metacyclic group can be presented on two generators and three defining relations as follows (see~\cite[Theorem 7.21]{zassenhaus}):
$$G=\langle x, y\, |\, x^k=y^l, y^m=1, x^{-1}yx=y^n\rangle, $$
where $k, l, n,  m$ are positive integers such that $m\mid (n^k-1)$ and $m\mid l(n-1)$.

It follows from the above result that metacyclic groups are 2-generated. 

\paragraph*{\textbf{Minimum Generating Set of a Group}}
Let $G$ be a finite group. A generating set of $G$ with minimum size is called \emph{a minimum generating set}. The size of a minimum generating set of a group $G$ is denoted by $d(G)$. We mention a few results on $d(G)$, which will be required later in the paper.

\begin{theorem}[Theorem 1.3,~\cite{dg}]\label{thm:Andrea}
Let $N$ be a minimal normal subgroup of a group $G$,  then $d(G/N)\leq d(G)\leq d(G/N)+1$.
\end{theorem}  
\begin{theorem}[Theorem 1.1,~\cite{lucchini_old}]\label{thm:unique}
    If $G$ is a non-cyclic finite group with a unique minimal normal subgroup $N$, then $d(G)= \max\{2, d(G/N)\}$.
\end{theorem}
\begin{theorem}[Satz 2,~\cite{gaschutz}]\label{thm:Andrea2}
Suppose $N$ is an abelian minimal normal subgroup of $G$, then $d(G)=d(G/N)+1$ iff $N$ is complemented in $G$ and the number of complements is $|N|^{d(G/N)}$.
\end{theorem}
We use the following easy lemma a few times in this paper.
\begin{lemma}\label{thm:Cornell}
Let $G$ and $H$ be groups with $\gcd(|G|,|H|)=1$, then $d(G\times H)= \max\{d(G), d(H)\}$.
\end{lemma}
\paragraph*{\textbf{Groups of Cube-Free Order}} We present two structural results concerning groups whose order is cube-free. These include a classification of non-solvable groups of cube-free order and a structural description of odd cube-free order groups. Such groups have been well-studied in the literature (see \cite{diet,qiao,DietrichWilson}). 

\begin{theorem}[Theorem 3.10,~\cite{qiao}]\label{thm:unsolve}
A non-solvable group $G$ of cube-free order has the form
$$G=\mathrm{PSL}(2, p)\times L,$$
where $p$ is some suitable prime,  and $L$ is of odd order.
\end{theorem}

\begin{cor}[Corollary 3.4,~\cite{qiao}]\label{cor:tower}
Let $G$ be a group of odd cube-free order. Then
$$G \cong S_{p_1} \rtimes \left( S_{p_2} \rtimes \left( \cdots \rtimes \left( S_{p_{k-1}} \rtimes S_{p_k} \right)\right) \right)
,$$ where $p_1 > p_2 > \cdots > p_k$ are the distinct primes dividing $|G|$.
\end{cor}

\section{\texorpdfstring{$2$}{}-Generated Numbers and Groups of Cube-Free Order}\label{sec:cube}
Recall that a positive integer $n$ is called a 2-generated number if every group in $\mathcal{G}_n$ can be generated by two elements (\Cref{def-2Generation}). In this section, we first observe that if $n$ is a $2$-generated number, then $n$ must be cube-free (\Cref{pro:cubefree}). We then determine a necessary and sufficient condition under which a cube-free number of the form  $p^{a}q^{b}$, with $p$ and $q$ being distinct primes, is $2$-generated.

\begin{pro}\label{pro:cubefree}
If a positive integer $n$ is a $2$-generated number, then $n$ is cube-free.
\end{pro} 
\begin{proof}
Suppose, in contrast, that $n$ is not cube-free. Let $p$ be a prime such that $p^k\mid n$ and $k\geq 3$. Note that $d(\Z_p\times \Z_p\times \Z_p)=3$. Consider a group $G=(\Z_p\times \Z_p\times \Z_p)\times \Z_{n/p^3}$. Note that $G$ is of order $n$, and $d(G) \geq 3$. Thus $n$ is not a $2$-generated number.
\end{proof}

\begin{obs}\label{cor:d3}
Let $G$ be a group of cube-free order. Then $d(G)\leq 3$.  In particular, square-free groups are 2-generated.
\end{obs}
\begin{proof}
Note that any Sylow $p$-subgroup of a group $G$ with cube-free order is cyclic or isomorphic to $\Z_p\times \Z_p$ for some prime $p$. By~\cite[Theorem 1]{lucchini}, we have $d(G)\leq \max_{p\, \, prime}d(S_p)+1$, where $S_p$ denotes a Sylow $p$-subgroup of $G$. Thus, the result follows.
\end{proof}
To study when a cube-free positive integer of the form $n = p^a q^b$, where $p$ and $q$ are distinct primes, is a 2-generated number, we start by considering the case when $a + b \leq 2$. In this case, $n$ is either a prime power or a square-free number. Every group of order $p^2$ is isomorphic to either $\mathbb{Z}_{p^2}$ or $\mathbb{Z}_p \times \mathbb{Z}_p$, both of which are 2-generated. Also, as noted in Observation~\ref{cor:d3}, groups of square-free order are 2-generated.
Now, in the following discussion, we will consider the cases when $n$ satisfies $a + b=3$.
\subsection{Groups of Order \texorpdfstring{$pq^2$}{} with \texorpdfstring{$p>q$}{}}
In this section, we prove that every group of order $pq^2$, where $p>q$ and $(p, q) \neq (3,2)$, is metacyclic. Metacyclic groups have been completely classified in~\cite{meta}.

\begin{theorem}\label{thm:pq}
         Let $n=pq^2$, where $p>q$ and $(p,q)\neq (3,2)$. Then every group of order $n$ is metacyclic.
    \end{theorem}

\begin{proof}
Let $G$ be a finite group of order $pq^2$, where $p>q$ are distinct primes and $(p, q) \ne (3, 2)$. We aim to show that $G$ is metacyclic. 

If $G$ is abelian, then either $G\cong \mathbb{Z}_{pq^2}$ or $G\cong \mathbb{Z}_{pq}\times \mathbb{Z}_q$. Hence $G$ is metacyclic. 

Assume now that $G$ is non-abelian. By Burnside's $p^aq^b$ Theorem, $G$ is solvable. Thus any minimal normal subgroup $N$ of $G$ is an elementary abelian $r$-group for some prime $r$ dividing $pq^2$, and therefore $
N\cong \Z_p,\quad \Z_q,\quad \text{or}\quad \Z_q\times \Z_q.
$ Suppose first that $N$ is cyclic. The conjugation action of 
$G$ on $N$ induces a natural homomorphism and gives an embedding
\[
G/C_G(N) \hookrightarrow \Aut(N)
,\] where $\Aut(N)$ is cyclic. So we conclude that $G/C_G(N)$ is cyclic. It is not difficult to show that $C_G(N)$ has to be cyclic, and so $G$ is metacyclic.

Now we assume $N\cong \Z_q\times \Z_q$. Then $N=C_G(N)$ because $G$ is non-abelian. Therefore $G/N\cong \mathbb{Z}_p$, and $G/N$ embeds in $\mathrm{GL}(2,q)$. Hence $p\mid (q^2-1)$. Since $p>q$, it follows that $p\mid (q+1)$, forcing $p=q+1$. Thus $(p,q)=(3,2)$, but this case is excluded by the hypothesis. Therefore, this situation cannot occur. Thus the only possible case is that $G$ has a cyclic normal subgroup with cyclic quotient, and hence $G$ is metacyclic.
\end{proof}
We now check the groups of order $n=pq^2$ with $p=3$ and $q=2$, that is, $n=12$.
\begin{rem}\label{rem:12}
Every group of order $12$ is $2$-generated.
\end{rem}
\begin{proof} 
Every group of order 12 is 2-generated. Note that every group of order 12 is isomorphic to one of the groups $\Z_{12}$,  $\Z_2\times \Z_2 \times \Z_3$, 
$A_4$,  $D_6$,  or the non-trivial semidirect product $Z_3\rtimes Z_4$. Thus, there are exactly 5 groups up to isomorphism (see~\cite[Theorem 4.24]{rotman}, for details). Clearly, they are all 2-generated.
\end{proof}

Based on the above discussion,  we deduce the following result.
\begin{theorem}\label{thm:pq^2}
    Let $n=pq^2$ where $p, q$ are primes and $p>q$. Then $n$ is a 2-generated number.
\end{theorem}
\begin{proof}
Since metacyclic groups are 2-generated, the desired result follows directly from Theorem~\ref{thm:pq} and Remark~\ref{rem:12}.
\end{proof}

\subsection{Groups of Order \texorpdfstring{$p^2q$}{} with \texorpdfstring{$p>q$}{}}
\begin{theorem}\label{thm:np^2q}
 Let $n=p^2q$ where $p, q$ are primes and $p>q$. Then $n$ is a 2-generated number if and only if $q\nmid (p-1)$.   
\end{theorem}
\begin{proof}
Suppose that $q\nmid (p-1)$. We will show that $n$ is a $2$-generated number. 
Note that if $G$ is an abelian group, then it is always $2$-generated. 
So, we can assume that $G$ is a non-abelian group of order $n$. 
Let $N$ be a minimal normal subgroup of $G$. 
We follow the same argument as in the proof of Theorem~\ref{thm:pq}. 
First note that $N$ cannot be cyclic, because in that case $|G/C_G(N)|$ divides $q-1$, so we get a contradiction as $p>q.$
Thus, we are only left with the case where $N = C_G(N) \cong \Z_p \times \Z_p$ 
(otherwise $G$ is abelian). 
Hence $G/C_G(N) \cong \Z_q$. 
Now we can apply~\Cref{thm:Andrea} to $G/N$, and we obtain $d(G)\leq 2$. 
Therefore $G$ is $2$-generated. 
Thus we conclude that any group of order $n$ is $2$-generated.

We now consider the other direction. Namely, we assume that $q|(p-1)$ and show that $n$ is not a 2-generated number by constructing a group $G$ of order $n$ which is not $2$-generated.
Since $q|(p-1)$, there is an element $a\in \Z_p^*$ of order $q$. We now define a homomorphism $\phi:\Z_q\longrightarrow \Aut(\Z_p\times \Z_p)$ as follows: Let $z\in \Z_q$, the image of $z$ under $\phi$, denoted $\phi_z$ maps $x\in \Z_p\times \Z_p$ to $a^zx$. One can easily verify that $\phi_z$ is an automorphism of  $\Z_p\times \Z_p$ and $\phi$ is an injective homomorphism.

Let $G=(\Z_p\times \Z_p)\rtimes_{\phi}\Z_q$. It is easy to check that an element $(x,0)$ has order $p$, where $x$ is a non-identity element of $\Z_p\times \Z_p$ (here $0$ denotes the identity in $\Z_q$). For an integer $s> 0$, we note that $(x,z)^s=((1+a^z+a^{2z}+\ldots+a^{(s-1)z})x,sz)$. We observe that 
\begin{equation}\label{eq-nil}
    (1-a^z)(1+a^z+a^{2z}+\ldots+a^{(s-1)z})=1-a^{sz}.
\end{equation}
 If $z\neq 0$, then for $s$ to be the order of $(x,z)$, we need $s$ to be a multiple of $q$. On the other hand, with $s=q$, the right-hand side of \Cref{eq-nil} is $0 \mod p$. Also as $(1-a^z)\not\equiv 0 \mod p$, $(1+a^z+a^{2z}+\ldots+a^{(s-1)z})\equiv 0 \mod p$. Thus, if $z$ is non-zero, then $(x,z)$ has order $q$ for all $x\in \Z_p\times \Z_p$ and if $x$ is non-identity then $(x,0)$ has order $p.$ These observations imply that
 \begin{description}
     \item[(a)] there are no elements of order $pq$ in $G$,
     \item[(b)] any subgroup  $P$ of order $p$ in $G$ has to be cyclic and normal.
 \end{description}
 Now we will show that $G$ is not 2-generated. Elements of order $p$ in $G$ will be elements of the unique Sylow $p$-subgroup of $G$ which is isomorphic to $\Z_p\times \Z_p.$ Thus no two elements of order $p$ can generate $G.$ If $o(g_1)=p$ and $o(g_2)=q$, then $\langle g_1\rangle \lhd G$, and therefore $|\langle g_1,g_2\rangle|=|\langle g_1\rangle \langle g_2\rangle|=pq$. Hence, they cannot generate $G$.
 
Lastly, we suppose that $o(g_1)=q$ and $o(g_2)=q$. Here we only need to consider the case when $\langle g_1\rangle \neq \langle g_2 \rangle $. Based on the preceding discussion, we must have $g_1=(x_1,z_1)$ and $g_2=(x_2,z_2)$, for some $x_1,x_2\in\Z_p\times \Z_p$ and $z_1,z_2\in \Z_q\setminus\{0\}$. Since $z_2$ is non-zero, it generates $\Z_q$ and so $-z_1=rz_2$ for a suitable $r.$ Note that $\langle g_2\rangle=\langle g_2^r\rangle.$ Let $g=g_1g_2^r$. Then $\langle g_1, g_2\rangle= \langle g_1, g\rangle,$ where $o(g)=p$. However, we have already seen that an element of order $p$ and an element of order $q$ cannot generate $G$.
\end{proof}
\section{Main Results}\label{sec:main}
In this section, we give a proof of \Cref{thm:odd} and \Cref{thm:even}.
\mainthmOdd*

\begin{proof}
Suppose that $n=p_1^{e_1}p_2^{e_2}\dots p_k^{e_k}$ is an odd positive integer and let $\pi(n) = \{p_1, p_2, \dots, p_k\}$. 

In the forward direction, we assume that $n$ is a 2-generated number. It is clear that $n$ has to be a cube-free number (see Proposition~\ref{pro:cubefree}). 

Suppose, for the sake of contradiction, that there are primes $p, q \in \pi(n)$ such that $p>q$ and $p^2 \mid n$ and $q \mid (p - 1)$. Then we can construct a group $G$ of the form $H \times \mathbb{Z}_{n/(p^2q)}$, where $|H| = p^2q$ and $d(H) = 3$. The existence of such a group $H$ is guaranteed by Theorem~\ref{thm:np^2q}, since $q\mid (p-1)$. Thus, $d(H \times \mathbb{Z}_{n/p^2q}) \geq 3$, contradicting the assumption that $n$ is 2-generated. This completes the proof of necessity. 

To prove the converse, we use induction on $k$. For the base case $k=1$, the statement holds trivially.
The inductive case is addressed by considering the following two cases:
\begin{description}
    \item[(i)] $e_1 = 2$;
    \item[(ii)] $e_1 = 1$.
\end{description}

\textbf{Case (i):} Suppose $n = p_1^{2}p_2^{e_2}\dots p_k^{e_k}$ is cube-free and  for every $i>j$, $p_j\nmid (p_{i}-1)$  or $e_i=1$.
Since $n$ is an odd and cube-free number, by Corollary~\ref{cor:tower}, the group $G$ is of the form $S_{p_1} \rtimes_{\phi} H$, where $S_{p_1}$ is the Sylow $p_1$-subgroup of $G$, $H$ is a subgroup of order $n/p_1^2$, and $\phi$ is a homomorphism from $H$ to $\Aut(S_{p_1})$.

If $\phi$ is the trivial homomorphism, then $G \cong S_{p_1} \times H$. Note that $\gcd(p_1, |H|) = 1$. Since $H$ has $k-1$ distinct prime divisors, by the induction hypothesis, we have $d(H)\leq 2$. Hence, we can apply Lemma~\ref{thm:Cornell} to obtain 
\[
d(G) = \max\{d(S_{p_1}), d(H)\} \leq 2.
\]
Therefore, $G$ is 2-generated.

Now we assume that $\phi$ is a non-trivial homomorphism from $H$ to $\Aut(S_{p_1})$. We first consider that $S_{p_1} \cong \mathbb{Z}_{p_1^2}$, 
so $\Aut(\mathbb{Z}_{p_1^2}) \cong U(p_1^2)$, where $U(m)$ denotes the group of units of $\mathbb{Z}_m$. Note that $|U(p_1^2)| = p_1(p_1 - 1)$. Since the action $\phi$ is non-trivial, we have $\{1\} \neq \phi(H) \leq U(p_1^2)$, which implies $|\phi(H)| \mid p_1(p_1 - 1)$. However, this contradicts the assumption that $p_j \nmid (p_1 - 1)$ for all $p_j \in \pi(n) \setminus \{p_1\}$. Therefore, such a non-trivial homomorphism cannot exist if $S_{p_1} \cong \mathbb{Z}_{p_1^2}$. 

Now let $S_{p_1} \cong \mathbb{Z}_{p_1} \times \mathbb{Z}_{p_1}$. Suppose that there exists a non-trivial homomorphism $\phi : H \rightarrow \Aut(\mathbb{Z}_{p_1} \times \mathbb{Z}_{p_1}) \cong \mathrm{GL}(2, p_1)$. Since $H$ is of odd order with the condition that for every $p_j \in \pi(n)$, $p_j \nmid (p_1 - 1)$, from the proof of Lemma 2.1 in \cite{qiao}, we obtain that $\phi(H) \cong \mathbb{Z}_b$ for some $b \mid (p_1 + 1),$ $b \mid |H|$ and $\phi(H)$ is an irreducible subgroup of $\operatorname{GL}(2, p_1)$.

Since $H $ acts on $ S_{p_1} $ via the homomorphism $ \phi: H \to \operatorname{GL}(2, p_1) $, and $ \phi(H) $ acts irreducibly on $ S_{p_1} $ (see \cite[Lemma 2.1]{qiao}), it follows that the action of $ H $ on $ S_{p_1} $ is irreducible. That is, $ S_{p_1} $ has no proper non-trivial $ H $-invariant subgroups. Hence, $ S_{p_1} $ is a minimal normal subgroup of $ G $.
Note that $|G/S_{p_1}|$ has fewer than $k$ prime divisors, so we can apply the induction hypothesis to the quotient group, which implies that $d(G/S_{p_1}) \leq 2$. Suppose $d(G/S_{p_1}) = 1$. Since $S_{p_1}$ is a minimal normal subgroup of $G$, by Theorem~\ref{thm:Andrea}, we have $d(G) \leq 2$. Suppose $d(G/S_{p_1}) = 2$. By Theorem~\ref{thm:schur}, all the complements of $S_{p_1}$ in $G$ are conjugate to $H$. The number of such conjugates is equal to $[G : N_G(H)]$, and since $H \leq N_G(H)$, it follows that $[G : N_G(H)] \leq p_1^2$. Thus
\[
|S_{p_1}|^{d(G/S_{p_1})} = p_1^4 > [G : N_G(H)],
\]
and applying Theorem~\ref{thm:Andrea2}, we obtain $d(G) = 2$. Therefore, $G$ is 2-generated.

\bigskip
\textbf{Case (ii):} Let $e_1 = 1$. Then $S_{p_1} \cong \mathbb{Z}_{p_1}$. By Corollary~\ref{cor:tower}, we can write $G \cong \mathbb{Z}_{p_1} \rtimes_{\phi} H$, where $H$ is a group of cube-free order with $\gcd(p_1, |H|) = 1$. 
By applying the induction hypothesis to $H$ we obtain $d(G/S_{p_1}) = d(H) \leq 2$.

First, suppose $\phi$ is the trivial homomorphism. Then, as in the earlier case, applying Lemma~\ref{thm:Cornell} we conclude that $G$ is 2-generated.

Now, suppose $\phi : H \rightarrow \mathbb{Z}_{p_1}^*$ is a non-trivial homomorphism. If $H$ is cyclic, then $G$ is a metacyclic group, and therefore 2-generated.

Next, we assume that $H$ is non-cyclic, so $d(G/S_{p_1}) = d(H) = 2$. Then $\ker(\phi) \neq \{1\}$, and $S_{p_1} \cong \mathbb{Z}_{p_1}$ is a minimal normal subgroup of $G$. By Theorem~\ref{thm:schur}, the complements of $S_{p_1}$ are precisely the conjugates of $H$. As in the earlier case, using Theorem~\ref{thm:Andrea2}, we get $d(G) = 2$. 
\end{proof}
Next, we address the case where $n$ is an even positive integer and examine the conditions under which $n$ is a 2-generated number. First we recall Theorem~\ref{thm:pq^2}, and Theorem~\ref{thm:np^2q}, which state that  $n=4p$ is a 2-generated number, whereas $n=2p^2$ is not a 2-generated number, for any odd prime $p$, respectively. 

\mainthmEven*
\begin{proof}

Let $n = 2^\alpha p_1^{e_1} p_2^{e_2} \cdots p_k^{e_k}$, where each $p_i$ is an odd prime and $\alpha, e_i$ are positive integers. 

Suppose first that $n$ is a $2$-generated number. Then by Proposition~\ref{pro:cubefree}, $n$ must be \emph{cube-free} and thus $\alpha \leq 2$. We aim to show that each exponent $e_i = 1$.

Suppose, for contradiction, that $e_i = 2$ for some $i$. Then $2p_i^2 \mid n$. Consider a group $H$ of order $2p_i^2$ which is not 2-generated (such a group exists by \Cref{thm:np^2q}). Define $G = H \times \mathbb{Z}_{n/(2p_i^2)}$. Since $d(H)=3$, we have $d(G)\geq 3$, contradicting the assumption that $n$ is a $2$-generated number. Thus, each $e_i = 1$. This proves that if $n$ is a $2$-generated number, then $n = 2^\alpha p_1 \cdots p_k$ with $\alpha \leq 2$ and $p_i$'s are distinct odd primes.

For the other direction, suppose $n = 2^\alpha p_1 \cdots p_k$ with $\alpha \leq 2$. Let $G$ be a group of order $n$. We now show that $n$ is a 2-generated number. If $\alpha = 1$, then $n$ is square-free. All such $G$ are 2-generated.

We now proceed with the remainder of the proof under the assumption that $\alpha = 2$, considering the following two cases:

\noindent \textbf{\textbf{Case 1:}} $G$ is non-solvable. By Theorem~\ref{thm:unsolve}, any non-solvable group of cube-free order has the form
\[
G = \mathrm{PSL}(2, p) \times L,
\]
where $p \geq 5$ is congruent to $3$ or $5$ modulo $8,$ $|PSL(2,p)|/ 8$ is an odd square-free number, and $L$ is a group of odd square-free order. One can check that to match the order of the group, we also require $p$ to be congruent to 3 or 5 modulo 8. Note that $\gcd(|\mathrm{PSL}(2, p)|, |L|) = 1$ as the odd factor of $|G|$ is square-free. Moreover, as both $\mathrm{PSL}(2, p)$ and $L$ are 2-generated (the former is non-abelian simple, the latter square-free), it follows from Lemma~\ref{thm:Cornell} that $d(G) = \max\{d(\mathrm{PSL}(2, p)), d(L)\} = 2$. Thus, in this case, $G$ is 2-generated.

\noindent \textbf{Case 2:} $G$ is solvable. We analyze this case based on the structure of the Sylow 2-subgroup $S_2$. There are two possibilities.

First, we suppose that $S_2 \cong \mathbb{Z}_4$. Then all Sylow subgroups of $G$ are cyclic, since all odd primes appear only to the first power. Thus, $G$ is a \emph{Z-group}. By \cite[Theorem 10.1.10]{robin}, Z-groups are 2-generated.

Now we assume that $S_2 \cong \mathbb{Z}_2 \times \mathbb{Z}_2$. In this case, we proceed by  \emph{induction on $k$}, the number of distinct odd prime divisors of $n$.

If $k = 1,$ then $n = 4p_1$, and by Theorem~\ref{thm:pq^2}, such groups are always 2-generated.

Let $G$ be a group of order $4p_1 \cdots p_k$. Since $G$ is solvable, it contains a \emph{non-trivial minimal normal subgroup $N$}, which must be elementary abelian.

Suppose $N \cong \mathbb{Z}_{p_i}$ for some odd prime $p_i$.   
Then $G/N$ has order $4p_1 \cdots p_{i-1} p_{i+1} \cdots p_k$, and by the  induction hypothesis, it is 2-generated. If $d(G/N) = 1$, then by Theorem~\ref{thm:Andrea}, $G$ is 2-generated. If $d(G/N) = 2$, then note that $\gcd(|N|, |G/N|) = 1$, so by the Schur--Zassenhaus theorem  (see \Cref{thm:schur}), $N$ is complemented in $G$. Using Theorem~\ref{thm:Andrea2} we can conclude that $d(G) = 2$.

On the other hand, when $N \cong \mathbb{Z}_2$ or $\mathbb{Z}_2 \times \mathbb{Z}_2$, and $G$ has no minimal normal subgroup corresponding to any odd prime. In this case, $N$ must be the \emph{unique} minimal normal subgroup of $G$. The quotient $G/N$ is of square-free order, which is less than that of $G$; therefore, $d(G/N) \leq 2$. Then by Theorem~\ref{thm:unique}, we have
\[
d(G) = \max\{2, d(G/N)\} = 2.
\]
\end{proof}

\section*{Acknowlegdements}
The authors sincerely thank Dr.~A.~Satyanarayana Reddy for introducing them to the concept of \(P\)-numbers. Kavita Samant gratefully acknowledges Indian Institute of Technology Gandhinagar for its support during her visit, which facilitated the initial discussions with the coauthors. Dhara Thakkar is supported by the JSPS KAKENHI grant No. JP24H00071.

 \section*{Data availability}
No data was used for the research described in the article.

\addcontentsline{toc}{section}{Bibliography}
\bibliographystyle{acm}
\bibliography{cube}
\end{document}